\documentclass[12pt,a4paper,reqno]{amsart}

\usepackage[latin1]{inputenc}
\usepackage{mathptmx}
\usepackage{amssymb}
\usepackage{hyperref}

\usepackage{pgf, tikz}
\usepackage{color}
\usepackage{algorithm}
\usepackage[noend]{algpseudocode}
\newcommand*\Let[2]{\State #1 $\gets$ #2}

\algrenewcommand\algorithmicrequire{\textbf{Input:}}
\algrenewcommand\algorithmicensure{\textbf{Return:}} 
\usepackage{url}

\definecolor{light}{gray}{0.9}
\definecolor{medium}{gray}{0.8}

\newtheorem{theorem}{Theorem}
\newtheorem{lemma}[theorem]{Lemma}

\newtheorem{proposition}[theorem]{Proposition}

\theoremstyle{definition}
\newtheorem{remark}[theorem]{Remark}
\newtheorem{definition}[theorem]{Definition}

\numberwithin{equation}{section}

\def\ZZ{{\mathbb Z}}

\def\RR{{\mathbb R}}

\def\CC{{\mathbb C}}

\def\Hilb{\operatorname{Hilb}}

\def\gp{\operatorname{gp}}

\def\id{\operatorname{id}}

\def\conv{\operatorname{conv}}
\def\codim{\operatorname{codim}}

\def\hht{\operatorname{ht}}
\def\lcm{\operatorname{lcm}}
\def\vol{\operatorname{vol}}
\def\para{\operatorname{par}}
\def\Ker{\operatorname{Ker}}
\def\rec{\operatorname{rec}}
\def\lev{\operatorname{lev}}

\def\detsum{\operatorname{detsum}}
\def\Cl{\operatorname{Cl}}

\let\epsilon=\varepsilon

\allowdisplaybreaks

\begin{document}
\title{Normaliz 2013--2016}
\author[W. Bruns]{Winfried Bruns}
\address{Winfried Bruns\\ Universit\"at Osnabr\"uck\\ FB Mathematik/Informatik\\ 49069 Osna\-br\"uck\\ Germany}
\email{wbruns@uos.de}
\author[R. Sieg]{Richard Sieg}
\address{Richard Sieg\\ Universit\"at Osnabr\"uck\\ FB Mathematik/Informatik\\ 49069 Osna\-br\"uck\\ Germany}
\email{risieg@uos.de}
\author[C. S\"oger]{Christof S\"oger}
\address{Christof S\"oger\\ Alter M\"hlenweg 1\\ 49504 Lotte\\ Germany}
\email{csoeger@uos.de}



\thanks{The second author was partially supported by the German Research Council DFG-GRK~1916 and a doctoral fellowship of the German Academic Exchange Service. }
\begin{abstract}
In this article we describe mathematically relevant extensions to Normaliz that were added to it during the support by the DFG SPP ``Algorithmische und Experimentelle Methoden in Algebra, Geometrie und Zahlentheorie'': nonpointed cones, rational polyhedra, homogeneous systems of parameters, bottom decomposition, class groups and systems of module generators of integral closures.
\end{abstract}

\maketitle

\section{Introduction}

The software package Normaliz \cite{Nmz} has been developed by the algebra and discrete mathematics group at Osnabrück since 1998. It is a tool for the computation of lattice points in rational polyhedra.
\begin{figure}[hbt]
	\begin{center}
		\begin{tikzpicture}[scale=0.5]
		
		\filldraw[yellow] (5,5.2) -- (4,5.2) -- (4,5.2) -- (-1,3) -- (-2,-1) -- (2,-4.2) -- (4.5,-4.2) -- cycle;
		\draw (4,5.2) -- (-1,3) -- (-2,-1) -- (2,-4.2);
		
		\foreach \x in {-4,-2,0,2,4}
		\foreach \y in {-3,-1,1,3}
		{
			\filldraw[fill=black] (\x,\y)  circle (1.5pt);
		}
		\foreach \x in {-3,-1,1,3}
		\foreach \y in {-4,-2,0,2,4}
		{
			\filldraw[fill=black] (\x,\y)  circle (1.5pt);
		}
		\end{tikzpicture}
	\end{center}
	\caption{Lattice points in a polyhedron}\label{nonpfig}
\end{figure}
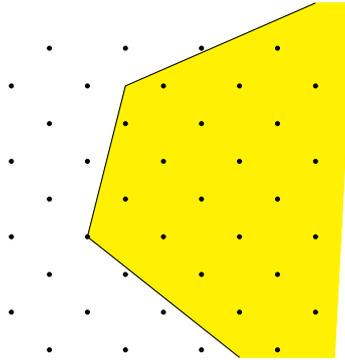
Meanwhile it has been cited about 130 times in the literature (see \cite{Nmz}) with applications to algebraic geometry, commutative algebra, polytope theory, integer programming, combinatorial topology, group theory, theoretical physics and other areas. There exist interfaces to the major computer algebra systems CoCoA \cite{CoCoA}, \cite{ABS}, GAP \cite{GAP-NmzInterface}, Macaulay2 \cite{M2} and Singular \cite{Sing} and to polymake \cite{pmake}, a comprehensive tool for the computation of polytopes. Normaliz is used by other packages, notably by Regina \cite{Reg}, a tool for the exploration of $3$-manifolds, and by SecDec \cite{SecDec} in the computation of multiscale integrals.

During the second half of the SPP 1489 ``Algorithmische und Experimentelle Methoden in Algebra, Geometrie und Zahlentheorie'' Normaliz was supported by the SPP. In this article we want to give an overview of those developments during the period of support that concern important mathematical aspects. For the mathematical background and unexplained terminology we refer the reader to Bruns and Gubeladze \cite{BG}.

The main algorithms of Normaliz have been documented in the papers by Bruns with Koch \cite{BK}, Ichim \cite{BI}, Hemmecke, Ichim, Köppe and Söger \cite{BHIKS}, Söger \cite{BS} and Ichim and Söger \cite{BIS}. See \cite{BIS} for the performance of Normaliz on its main tasks.

Let $A$ be a $e\times d$ matrix with integer entries, and $a\in \ZZ^e$. Then the set
\begin{equation}
P=\{x\in \RR^d; Ax\ge a\}\label{basic-p}
\end{equation}
is called a \emph{rational polyhedron}. Moreover, let $B$ be a $f\times d$ matrix of integers, $b\in\ZZ^f$, $C$ be a $g\times d$ matrix of integers and $c,m\in \ZZ^g$. Then 
\begin{equation}
L=\{x\in \ZZ^d: Bx=b,\ Cx\equiv c (m)\}\label{basic-l}
\end{equation}
is an \emph{affine sublattice} in $\RR^d$, and it is the task of Normaliz to ``compute''' the set $P\cap L$. So Normaliz can be considered as a tool for solving linear diophantine systems of inequalities, equations and congruences. Rational polyhedra and affine lattices can also be, and often are, described in terms of parametrizations or generators, and the conversion between the two descriptions for $P$ and $L$ separately is a basic task prior to the computation of $P\cap L$.

The main computation goals of Normaliz are
\begin{itemize}
\item[] \emph{Generation:} find a (finite) system of generators of $N=P\cap L$;
\item[] \emph{Enumeration:} Compute the Hilbert series
$$
H_N(t)=\sum_{x\in N} t^{\deg x}
$$
with respect to a grading on $\ZZ^d$.
\end{itemize}
Of course, \emph{Generation} must be explained, and \emph{Enumeration} must even be modified somewhat to make sense in the general case.

The core case for Normaliz computations is the homogeneous one, in which the vectors $a,b,c$ in \eqref{basic-p} and \eqref{basic-l} are $0$, under the additional assumption that $P$, which in the homogeneous case is a cone $C$, is pointed, i.e. it does not contain a nontrivial linear subspace.  The affine lattice $L$ is then a subgroup of $\ZZ^d$, and for \emph{Generation} Normaliz must compute a Hilbert basis of the monoid $M=C\cap L$, i.e., a minimal system of generators of the monoid $M$.

For a long time Normaliz could only handle homogeneous systems in the pointed case. These restrictions have been removed in two steps: version 2.11.0 (April 2014) introduced inhomogeneous systems and version 3.1.0 (February 2016) finally removed the condition that $P$ has a vertex.  These extensions will be discussed in Sections \ref{nonpoint} and \ref{inhom}, where also \emph{Generation} and \emph{Enumeration} will be made precise.

The Hilbert series is (the Laurent series expansion of) a rational function of type
$$
H_N(t)=\frac{Q(t)}{(1-t^{g_1})\cdots (1-t^{g_r})}
$$
with a Laurent polynomial $Q(t)\in \ZZ[t,t^{-1}]$. In the general case there is no canonical choice for the exponents $g_1,\dots,g_r$ in the denominator. One good possibility is to take them as the degrees of, in the language of commutative algebra, a homogeneous system of parameters (hsop). Such degrees can be found if one analyzes the face lattice of the recession cone of the system \eqref{basic-p} and \eqref{basic-l}; the cone of solutions of the associated homogeneous system. This approach will be developed in Section \ref{hsop}. The option to use an hsop was introduced in version 3.1.2 (September 2016).

The primal algorithm of Normaliz is based on triangulations. A critical magnitude for the algorithm is the sum of the determinants of the simplicial cones in the triangulation. Since version 3.0.0 (September 2015) this determinant sum can be optimized by using a bottom decomposition. In Section \ref{bottom} we explain how a bottom decomposition can be computed.

A normal affine monoid $M$ has a well-defined class group. By a theorem of Chouinard (see \cite[4.F]{BG}) it coincides with the class group of the monoid algebra $K[M]$ for an arbitrary field $K$. Since version 3.0.0 Normaliz computes the class group, as explained in Section \ref{class}.

The primal algorithm of Normaliz finds the Hilbert basis by first computing a system of generators of $M$ as a module (in a natural way) over an input (or precomputed) monoid $M_0$. Therefore it can be used to find a minimal system of module generators of $M$ over $M_0$. In more picturesque language these generators are called ``fundamental holes'' of $M_0$. See Kohl, Li, Rauh and Yoshida \cite{KLRY} for a package making use of this Normaliz feature.

There are several other extensions and options that have been introduced during the support of the Normaliz project by the SPP:
\begin{enumerate}
\item new input format (with backward compatibility),
\item standard sorting of vector lists in the output,
\item completely revised linear algebra with permanent overflow check,
\item automatic choice of integer type (64 bit or infinite precision),
\item computation of integer hulls as an option,
\item refinement of the triangulation to a disjoint decomposition,
\item subdivision of `` large'' simplicial cones by using SCIP \cite{Scip} or approximation methods (see Bruns, Sieg and Söger  \cite{ICMS16}),
\item a normality test that avoids the computation of the full Hilbert basis,
\item improvement of the Fourier-Motzkin algorithm in connection with pyramid decomposition (see \cite{BIS}),
\item revision of the dual algorithm,
\item various improvements in the algorithms that save memory and computation time,
\item improvements in NmzIntegrate (see Bruns and Söger \cite{BS}).
\end{enumerate} 
The file \texttt{CHANGELOG} in the Normaliz distribution gives an overview of the evolution.

\section{The Normaliz primal algorithm}\label{primal}

The heart of Normaliz are two algorithms. The \emph{primal algorithm} can be applied both to \emph{Generation} and \emph{Enumeration}. Among the two it is the considerably more complicated one. The \emph{dual algorithm} can only be used for \emph{Generation}. We refer the reader to Bruns and Ichim \cite{BI} for its description.

Since some details of the primal algorithm play a role in the following, we include a brief outline.
The primal algorithm starts from a pointed rational cone $C\subset\RR^d$ given by a system of generators $x_1,\dots,x_n$ and a sublattice $L\subset\ZZ^d$ that contains $x_1,\dots,x_n$. (Other types of input data are first transformed into this format.) The algorithm is composed as follows:
\begin{enumerate}
	\item Initial coordinate transformation to $E=L\cap (\RR x_1+\dots+\RR x_n)$;
	\item Fourier-Motzkin elimination computing the support hyperplanes of $C$;
	\item pyramid decomposition and computation of the lexicographic triangulation $\Delta$;
	\item evaluation of the simplicial cones in the triangulation:
	\begin{enumerate}
		\item enumeration of the set of lattice points $E_\sigma$ in the fundamental domain of a simplicial subcone $\sigma$,
		\item reduction of $E_\sigma$ to the Hilbert basis $\Hilb(\sigma)$,
		\item Stanley decomposition for the Hilbert series of $\sigma'\cap L$ where $\sigma'$ is a suitable translate of $\sigma$;
	\end{enumerate}
	\item Collection of the local data:
	\begin{enumerate}
		\item reduction of $\bigcup_{\sigma\in\Delta} \Hilb(\sigma)$ to $\Hilb(C\cap L)$,
		\item accumulation of the Hilbert series of the intersections $\sigma'\cap L$; 
	\end{enumerate}
	\item reverse coordinate transformation to $\ZZ^d$.
\end{enumerate}

The algorithm does not strictly follow this chronological order, but interleaves steps 2--5 in an intricate way to ensure low memory usage and efficient parallelization. The steps 2 and 5 are treated in \cite{BI}. Steps 3 and 4 are described in \cite{BIS}; the translates $\sigma'$ in 4c are chosen in such a way that $C\cap L$ is the disjoint union of their lattice points. 

In view of the initial and final coordinate transformations 1 and 6 it is no essential restriction to assume that $\dim C=d$ and $L=\ZZ^d$, as we will often do in the following.

\section{Nonpointed cones and nonpositive monoids}\label{nonpoint}

In this section we discuss only the homogeneous situation in which the polyhedron $P\subset\RR^d$ is a cone $C$ and the affine lattice $L$ is a subgroup of $\RR^d$. Since \cite{BG} contains an extensive treatment of the mathematical background, we content ourselves with a brief sketch and references to \cite{BG}.

The basic finiteness result in polyhedral convex geometry is the theorem of Minkowski-Weyl \cite[1.15]{BG}. It shows that one can equivalently describe cones by generators or by inequalities.

\begin{theorem}\label{MW}
The following conditions are equivalent for a subset $C$ of $\RR^d$:
\begin{enumerate}
\item there exist (integer) vectors $x_1,\dots,x_n$ such that $C=\RR_+x_1+\dots+\RR_+x_n$;
\item there exist linear forms (with integer coefficients) $\sigma_1,\dots,\sigma_s$ on $\RR^d$ such that $C=\{x\in\RR^d: \sigma_i(x)\ge 0,\ i=1,\dots,s\}$.
\end{enumerate}
\end{theorem}

With the additional requirement of integrality in the theorem, $C$ is called a \emph{rational cone}. If $\dim \RR C=d$ and  the number of linear forms is chosen to be minimal, the $\sigma_i$ in the theorem are uniquely determined up to positive scalars, and they are even unique if we additionally require that the coefficients are coprime integers. In this case we call $\sigma_1,\dots,\sigma_s$ the \emph{support forms} of $C$. The map $\sigma:\RR^d\to\RR^s, \sigma(x)=(\sigma_1(x),\ldots,\sigma_s(x)),$ is called the \emph{standard map} of $C$. Clearly, $\sigma$ maps $\ZZ^d$ to $\ZZ^s$ in the rational case.

The conversion from generators to inequalities in the description of cones is usually called \emph{convex hull computation} and the converse transformation is \emph{vertex enumeration}. These two transformations are two sides of the same coin and algorithmically completely identical since they amount to the dualization of a cone. While this is not the main task of Normaliz, it often outperforms dedicated packages. See the recent benchmarks by Assarf et al.\ \cite{pmk-bench} and Köppe and Zhou \cite{KZ}.

In view of the remarks in Section \ref{primal} we can assume that $\dim C=d$ and that the subgroup $L\subset \ZZ^d$ is $\ZZ^d$ itself. Thus the task is to compute the monoid $M=C\cap \ZZ^d$. The basic finiteness result for such monoids is \emph{Gordan's lemma} \cite[2.9]{BG}:

\begin{theorem}
There exist $x_1,\dots,x_n\in \RR^d$ such that $M=\ZZ_+x_1+\dots+\ZZ_+x_n$.
\end{theorem}

At this point it is useful to borrow some terminology from number theory. We call $U(M)=\{x\in M: -x\in M\}$ the \emph{unit group} of $M$. Clearly, $U(M)=\{x \in M: \sigma(x)=0\}$. The unit group is the maximal subgroup of $\ZZ^d$ that is contained in $M$. One calls $M$ \emph{positive} if $U(M)=0$. Similarly, the \emph{maximal linear subspace} of $C$ is $U(C)=\Ker\sigma$. It is not hard to see that the positivity of $M$ is equivalent to the \emph{pointedness} of $C$: one has $U(C)=\RR U(M)$, and therefore $U(M)=0$ if and only if $U(C)=0$.

An element $x\in M\setminus U(M)$ is called \emph{irreducible} if a decomposition $x=y+z$ with $y,z\in M$ is only possible with $y\in U(M)$ or $z\in U(M)$. The role of the irreducible elements in the generation of $M$ is illuminated by the following theorem \cite[2.14 and 2.26]{BG}.

\begin{theorem} Let $M=C\cap\ZZ^d$. Then the following hold:
\begin{enumerate}
\item every element $x$ of $M$ can be written in the form $x=u+y_1+\dots +y_m$ where $u$ is a unit and $y_1,\dots,y_m$ are irreducible;
\item up to differences by units, there exist only finitely many irreducibles in $M$;
\item let $H\subset M$; then the following are equivalent:
\begin{enumerate}
\item $M=U(M)+\ZZ_+H$ and $H$ is minimal with this property;
\item $H$ contains exactly one element of each residue class of irreducibles modulo $U(M)$.
\end{enumerate}
\item $M\cong U(M)\oplus \sigma(M)$. 
\end{enumerate}
\end{theorem}

If $H$ satisfies the equivalent conditions in statement 3 we call it a \emph{Hilbert basis} of $M$. Clearly, together with a basis of the free abelian group $U(M)$ the Hilbert basis gives a minimal finite description of $M$. Statement 4 shows that $U(M)$ and $\sigma(M)$ are independent of each other. Moreover, the submonoid of $M$ generated by $H$ is isomorphic to $\sigma(M)$. 

Note that $H$ is uniquely determined if $M$ is positive, so that we can denote it by $\Hilb(M)$.  In the general case $H$ is a Hilbert basis of $M$ if and only if $\sigma(H)=\Hilb(\sigma(M))$. Therefore \emph{Generation} can be split into two subtasks: (i) find $U(M)$, the kernel of the $\ZZ$-linear map $\sigma|\ZZ^d$ and (ii) find the Hilbert basis of the positive monoid $\sigma(M)$. The first task is a matter of solving a homogeneous diophantine system of linear equations, and the second is what Normaliz has done from its very beginnings. 

The theory above can be developed for arbitrary affine monoids; see \cite[Ch. 2]{BG}. However, the direct sum decomposition $M\cong U(M)\oplus\sigma(M)$ is not always possible.

What we have described for Hilbert bases, applies similarly to extreme rays of cones. These are only defined modulo $U(C)$ in the general case.

Normaliz' dual algorithm for the computation of Hilbert bases effectively does all its computations in in the pointed cone $\sigma(C)$; see \cite{BI}. Nevertheless, versions before 2.11.0 did not output the results if the cone was not pointed. 

The primal algorithm could have been modified for Hilbert basis computations of nonpointed cones, but we do not see a way for the computation of Hilbert series in the nonpointed case. Moreover, the passage to the quotient modulo the maximal linear subspace reduces the dimension and therefore speeds up the computation. Let us look at a simple example (Figure \ref{nonpfig}). The output shows:
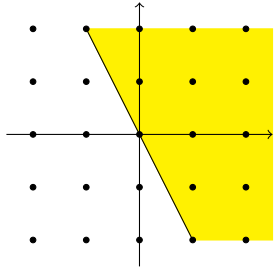
\begin{figure}[hbt]
\begin{center}
	\begin{tikzpicture}[scale=0.7]
	
	\filldraw[yellow] (1,-2) -- (-1,2) -- (2.5,2) -- (2.5,-2) -- cycle;
	
	\foreach \x in {-2,...,2}
	\foreach \y in {-2,...,2}
	{
		\filldraw[fill=black] (\x,\y)  circle (1.5pt);
	}
	\draw[->] (-2.5,0) -- (2.5,0);
	\draw[->] (0,-2.5) -- (0,2.5);
	\draw (1,-2) -- (-1,2);
	\end{tikzpicture}
\end{center}
\caption{A nonpointed cone}\label{nonpfig}
\end{figure}
\begin{quotation}
\begin{verbatim}
1 Hilbert basis elements of degree 1:
0 1

0 further Hilbert basis elements of higher degree:

1 extreme rays:
0 1

1 basis elements of maximal subspace:
1 -2
\end{verbatim}
\end{quotation}
	
Since in the vast majority of cases Normaliz is applied to positive monoids, Normaliz does not (always) try to compute $U(M)$ beforehand -- very likely it is $0$. The computation of $\sigma$ requires the computation of the support hyperplanes of $C$.  Eventually these will be known, but their computation is inevitably intertwined with the computation of the triangulation, and would essentially have to be done twice. Therefore Normaliz takes the following ``bold'' approach in the primal algorithm:
\begin{enumerate}
	\item Start the computation and proceed under the assumption that $C$ is pointed.
	\item As soon as the support hyperplanes have been computed, decide positivity.
	\item If it should fail, throw an exception, perform the coordinate transformation to the pointed quotient, and restart the computation.
\end{enumerate}

After \emph{Generation} let us discuss \emph{Enumeration}. A linear form $\deg:\ZZ^d\to\ZZ$ is called a \emph{grading} on $M$ if $\deg x\ge 0$ for all $x\in M$ and $\deg x >0$ for $x\in M\setminus U(M)$. The \emph{Hilbert series} of $M$ with respect to $\deg$ is the formal power series
$$
H_M(t)=\sum_{x\in M} t^{\deg x}.
$$
If $M$ is positive there exist only finitely many elements in each degree, and the definition of $H_M(t)$ makes sense. This is not the case if $U(M)\neq 0$ -- there exist already infinitely many elements of degree $0$. Hence, if $M$ is not positive, the only Hilbert series that we can associate to it, is that of $\sigma(M)$. In fact, since $\deg(x)=0$ for $x\in U(M)$, $\deg$ induces a grading on $\sigma(M)$: if $\sigma(x)=\sigma(y)$, then $x-y\in U(M)$, and so $\deg x=\deg y$. Therefore Normaliz (always) computes $H_{\sigma(M)}(t)$, and the invariants that depend on the Hilbert series are also computed for $\sigma(M)$.

\section{Inhomogeneous systems}\label{inhom}

In algebraic geometry one passes from an affine variety to a projective one by \emph{homogenization}, and the same technique is used in discrete convex geometry to reduce algorithms for polyhedra to algorithms for cones. Let $P\subset \RR^d$ be an arbitrary polyhedron. Then the \emph{cone over $P$} is the \emph{closed} set
$$
C(P)=\overline{\RR_+(P\times \{1\})} \subset \RR^{d+1}.
$$
This amounts to passing from an inhomogeneous system to a homogeneous one by introducing a homogenizing variable, the $(d+1)$th coordinate. Setting the homogenizing variable equal to $1$, we get the inhomogeneous system back. In fact, it is not hard to see that one obtains a system of inequalities for $C(P)$ by homogenizing such a system for $P$ and adding the inequality $x_{d+1}\ge 0$.

If we set the homogenizing variable equal to $0$ we obtain the \emph{associated homogeneous system}, and its solution set is called the \emph{recession cone} in our case:
$$
\rec(P)=\{x\in\RR^d: (x,0)\in C(P) \}.
$$

It is useful to introduce the \emph{level} of a point $x\in \RR^{d+1}$,
$$
\lev(x)=x_{d+1}.
$$ 

By (de)homogenizing the Minkowski-Weyl theorem \ref{MW} one arrives at \emph{Motzkin's theorem}; see \cite[1.27]{BG}:

\begin{theorem}\label{Motz}
Let $P$ be a nonempty subset of $\RR^d$. Then the following are equivalent:
\begin{enumerate}
\item $P$ is a polyhedron;
\item There exist a nonempty polytope $Q$ and a cone $C$ such that $P=Q+C$.
\end{enumerate}
\end{theorem}

A polytope is a bounded polyhedron; a special case of Theorem \ref{Motz} is Minkow\-ski's theorem: $P$ is a polytope if and only if $P$ is the convex hull of finitely many points. 

For the cone $C$ in the theorem one has no choice: $C=\rec(P)$. The polytope $P$ is unique only if it is chosen minimal and $\rec(P)$ is pointed. In this case it must be the convex hull of the vertices of $P$. In the general case the vertices, like the extreme rays of cones, are only defined modulo the maximal linear subspace $U(\rec(P))$. 

One can interpret Theorem \ref{Motz} as saying that polyhedra are finitely generated: $Q$ is the convex hull of finitely many points, and the cone $C$ is finitely generated. Finite generation holds also for lattice points, as we will see now.

In the same way as polyhedra, one homogenizes an affine lattice: from $L\subset \ZZ^d$ one passes to the subgroup $\overline L$ of $\ZZ^{d+1}$ generated by $L\times \{1\}$. Normaliz goes this way, and then reduces the situation to the case $\overline L=\ZZ^{d+1}$ by preliminary coordinate transformations. For simplicity we will therefore assume that $\overline L=\ZZ^{d+1}$.

We want to compute the set $N=P\cap \ZZ^d$. The homogenization of $N$ is the monoid $M=C(P)\cap\ZZ^{d+1}$.  By analogy with $\rec(P)$ we define the \emph{recession monoid}
$$
\rec(N)=\{x\in \ZZ^d: (x,0)\in M\}.
$$

\begin{theorem}\label{Modfin}
Suppose that $N\neq \emptyset$.
\begin{enumerate}
\item Then there exist finitely many lattice points $y_1,\dots,y_m\in N$ such that 
$$
N=\bigcup_{i=1}^m x_i+\rec(N).
$$
\item The number $m$ is minimal if and only if there exists a Hilbert basis $H$ of $M$ such that
$$
\{y_1,\dots,y_m \}=\{y\in\ZZ^d: (y,1)\in H \}
$$
\item If $H$ is a Hilbert basis of M, then $\{x\in \ZZ^d: (x,0)\in H \}$ is a Hilbert basis of $\rec(N)$.
\end{enumerate}
\end{theorem}

Part 1 is \cite[2.12]{BG}, and the statements about Hilbert bases are easy to prove. The theorem entitles us to call $N$ a \emph{finitely generated module} over $\rec(N)$. The computation goal \emph{Generation} can now be made precise in the inhomogeneous case as well: compute a Hilbert basis of $\rec(N)$ and a \emph{system of module generators} $\{y_1,\dots,y_m\}$. By the theorem it is enough to compute a Hilbert basis of $M$. However, it would be foolish to overlook the shortcut that is possible: all candidates $x\in M$ with $\lev(x)>1$ can be immediately discarded. This holds both for the primal and the dual algorithm of Normaliz. (The primal algorithm does only produce elements $x$ with $\lev(x)\ge 0$. For the dual algorithm that processes the inequalities defining $C(P)$ one must start with the inequality $\lev(x)\ge 0$.)

As a simple example we consider the polyhedron in Figure \ref{figpol}. 
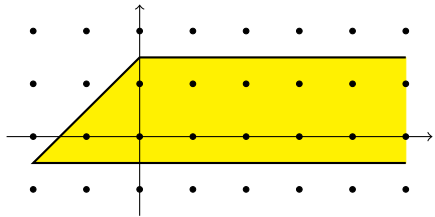
\begin{figure}[hbt]
\begin{center}
	\begin{tikzpicture}[scale=0.7]
	
	\filldraw[yellow] (5,-0.5) -- (-2,-0.5) -- (0,1.5) -- (5,1.5) -- cycle;
	
	\foreach \x in {-2,...,5}
	\foreach \y in {-1,...,2}
	{
		\filldraw[fill=black] (\x,\y)  circle (1.5pt);
	}
	\draw[->] (-2.5,0) -- (5.5,0);
	\draw[->] (0,-1.5) -- (0,2.5);
	\draw[thick] (5,-0.5) -- (-2,-0.5) -- (0,1.5) -- (5,1.5); 
	\end{tikzpicture}
\end{center}
\caption{A polyhedron in $\RR^2$}\label{figpol}
\end{figure}

Normaliz writes the results in homogenized coordinates:
\begin{quotation}
\begin{verbatim}
2 module generators:
-1 0 1
0 1 1

1 Hilbert basis elements of recession monoid:
1 0 0 
\end{verbatim}
\end{quotation}
The result can be checked by inspection.

The set $N$ has a \emph{disjoint} decomposition into residue classes mod $G=\gp(\rec(N))$ (where $\gp(M)$ is the group generated by $M$):
$$
N=\bigcup_{i=1}^r N_i, \qquad N_i\neq\emptyset,\quad N_i\cap N_j=\emptyset\text{ if } i\neq j,\quad x\equiv y\mod G \text{ for all }
x,y\in N_i.
$$
If $y_1,\dots,y_m$ is a system of module generators of $N$ as an $\rec(N)$-module, then obviously $r\le m$; in particular, $r$ is finite. It is justified to call $r$ the \emph{module rank} of $N$ over $\rec(N)$ because of the following functorial process. Let $K$ be a field and let $R=K[\rec(N)]$ be the monoid $K$-algebra defined by $\rec(N)$. Let $K[N]$ be the $K$-vector space with basis $N$. The ``multiplication''  $\rec(N)\times N\to N$, $(x,y)\mapsto x+y$ makes $K[N]$ a module over $R$ \cite[p. 51]{BG}. Since $R$ is an integral domain, $K[N]$ has a well-defined rank, which is exactly $r$, as one sees by passage to the field of fractions of $R$. An intermediate step of this passage is the Laurent polynomial ring $L=K[G]$, and we can get $K[N]\otimes_R L$ by introducing $K$-coefficients to $N+G$. This set decomposes into the subsets $N_i+G$, and one has $N_i+G=x+G$ for every $x\in N_i$. Therefore $K[N]\otimes_R L$ is the direct sum of $r$ free $L$-modules of rank $1$. In the example above, the module rank is $2$.

If $y_1,\dots,y_m$ have been computed, then it is very easy to find the module rank $r$: we simply count their pairwise different residue classes modulo $G$. But we can also compute $r$ as the number of lattice points in a polytope, and Normaliz resorts to this approach if a system of module generators is unknown. The polytope is a cross-section of $P$ with a complement of $\rec(P)$:

\begin{theorem}\label{modrank}
Let $z_1,\dots,z_s$ be a $\ZZ$-basis of $G=\gp(\rec(N))$. There exist $z_{s+1},\dots,\allowbreak z_d\in\ZZ^d$ such that $z_1,\dots,z_d$ is a $\ZZ$-basis of $\ZZ^d$. Set $H=\ZZ z_{s+1}+\dots+\ZZ z_d$, and  let $\pi:\RR^d\to \RR H$ denote the projection defined by $\pi|G=0$ and $\pi|H=\id_H$.

Then $\pi(P)$ is a (rational) polytope, and the module rank $r$ is the number of lattice points in $\pi(P)$.	
\end{theorem}

\begin{proof}
The first statement amounts to the existence of a complement $H$ of $G$ in $\ZZ^d$, i.e., a subgroup $H$ with $\ZZ^d=G+H$ and $G\cap H=0$. Such a complement exists if and only if $\ZZ^d/G$ is torsionfree. Let $z\in\ZZ^d$ such that $kz\in G$ for some $k\in\ZZ$, $k>0$. Since $G=\RR \rec(P)\cap \ZZ^d$, we must have $x\in G$.

The polyhedron $P$ is the Minkowski sum $Q+\rec(P)$ with a polytope $Q$. Since $\rec(P)\subset \RR G$, we have $\pi(P)=\pi(Q)$, and therefore $\pi(P)$ is a polytope. Clearly, the lattice points in $P$ are mapped to lattice points in $\pi(P)$, and two such points have the same image if and only they differ by an element in $G$. 

The only critical question is whether every lattice point in $\pi(P)$ is hit by a lattice point in $P$ by the application of $\pi$. There is nothing to show if $G=0$ since $\pi$ is the identity on $\RR^d$ then. So assume that $G\neq 0$. Let $p\in \pi(P)$ be a lattice point, $p=\pi(q)$ with $q\in P$. One has $\pi(p-q)=0$, and therefore $p-q\in \RR G$. Note that $\RR G=\RR \rec(P)$. In other words, $\rec(P)$ is a fulldimensional cone in $\RR G$. It contains a lattice point $x$ in its (relative) interior. Thus $(p-q)+kx\in \rec(P)$ for $k\in\ZZ$, $k\gg 0$, and $q+(p-q)+kx\in\ZZ^d$ is a preimage of $p$ in $P$ for $k\gg 0$.	
\end{proof}

Let us now discuss \emph{Enumeration} in the inhomogeneous case. As in the homogeneous case, we can only compute the Hilbert series of $N=P\cap\ZZ^{d+1}$ modulo $U(\rec(N))$. Therefore it is enough to discuss the case in which $\rec(P)$ or, equivalently, $C(P)$ is pointed.

Normaliz computes the Hilbert series via a \emph{Stanley decomposition}. This is a disjoint decomposition of the set of lattice points $P\cap \ZZ^d$ into subsets of the form 
$$
D=u+\sum_{i=1}^r \ZZ_+v_i
$$
 where $r$ varies between $0$ and $\dim P$ and $v_1,\dots,v_r$ are linearly independent. Provided $\deg v_i>0$ for $i=1,\dots,r$, the Hilbert series of $D$ is given by
\begin{equation}
H_D(t)=\frac{t^{\deg u}}{(1-t^{deg v_1})\cdots(1-t^{\deg v_r})}.\label{StD}
\end{equation}
In order to get the Hilbert series of $P\cap \ZZ^d$, it only remains to sum the Hilbert series of the components of the Stanley decomposition. 

In\cite{BIS} the computation of the Stanley decomposition in the homogeneous case is described in detail. Therefore we only discuss how to derive the Stanley decomposition of $P\cap\ZZ^d$ from a Stanley decomposition of $C(P)\cap \ZZ^{d+1}$. We must intersect all components of the Stanley decomposition of $C(P)\cap\ZZ^{d+1}$ with the hyperplane $L_1$ of level $1$ points. Since the levels of all participating vectors are integral and $\ge 0$, in a sum of level $1$ exactly one summand must have level $1$ and the others must have level $0$. 

\begin{proposition}\label{StDec}
Suppose that $C(P)$ is pointed, and that $D$ is a component in the Stanley decomposition of $C(P)$. Let $v_1,\dots,v_e$ be the vectors of level $1$ among $v_1,\dots,v_r$, and $v_{e+1},\dots,v_f$ those of level $0$. Then the following hold:
\begin{enumerate}
\item if $\lev(u)=1$, then $D\cap L_1=u+\sum_{i=e+1}^f \ZZ_+v_i$.
\item If $\lev(u)=0$, then $D\cap L_1$ is the disjoint union of the sets $u+v_j+\sum_{i=e+1}^f \ZZ_+v_i$, $j=1,\dots,e$ (and thus empty if $e=0$).
\item if $\lev(u)>1$, then $D\cap L_1=\emptyset$. 
\end{enumerate}
\end{proposition}

Note that $f-e \le \dim P$ if $D\cap L_1\neq \emptyset$. The proposition shows that the computation of a Stanley decomposition of $P\cap\ZZ^d$ is as easy (or difficult) as the computation for $C(P)\cap\ZZ^{d+1}$.

In the homogeneous case all degrees are nonnegative. In the inhomogeneous case this requirement would be an unnecessary restriction. Normaliz takes care of this aspect by computing a \emph{shift}. For our simple example above we obtain with $\deg(x_1,x_2)=x_1$
\begin{quotation}
\begin{verbatim}
Hilbert series:
1 1 
denominator with 1 factors:
1: 1  

shift = -1
\end{verbatim}
\end{quotation}
Thus the Hilbert series is
$$
t^{-1}\,\frac{1+t}{1-t}=\frac{t^{-1}+1}{1-t}.
$$

Normaliz lets the user specify a linear form $\delta$ that plays the role of the dehomogenization. This is already useful for compatibility with the input formats of other packages: often the first coordinate is used for (de)homogenization.

\begin{remark}
Inhomogeneous systems are often created by strict linear inequalities $\lambda(x)>0$ where $\lambda$ is linear (in addition to non-strict ones). These can be treated as inhomogeneous systems, but Normaliz also offers a variant called ``excluded faces''. Then homogenization (with its increase in dimension) is avoided at the expense of an inclusion-exclusion approach. This variant can also be used by NmzIntegrate.
\end{remark}

\section{Bottom decomposition}\label{bottom}

As mentioned above, Normaliz computes a triangulation of the cone $C$ whose rays are given by the input (or precomputed) system of generators, a partial triangulation for Hilbert bases and a full one for Hilbert series. 

The complexity of the Normaliz algorithm depends mainly on two parameters. The first is the size of the triangulation. The second is the determinant sum (or normalized volume) that determines the time needed for the evaluation of the simplicial cones in the triangulation. In the following $\vol$ denotes the $\ZZ^d$-normalized volume in $\RR^d$. It is the Euclidean volume multiplied by $d!$.

Let $\sigma$ be a simplicial cone generated by linearly independent vectors $v_1,\dots,v_d$. Then the normalized volume of the \emph{basic simplex} spanned by $0$ and $v_1,\dots,v_d$ is the absolute value of the determinant of the $d\times d$-matrix with rows $v_1,\dots,v_d$. Therefore we call it the \emph{determinant} $\det\sigma$ of $\sigma$. It is also the number of lattice points in the semi-open parallelotope
$$
\para(v_1,\dots,v_d)=\bigl\{a_1v_1+\sum+a_dv_d: 0\le a_i<1,\ i=1,\dots,d\bigr\},
$$
which is also referred as the \emph{fundamental domain } of $\sigma$. Normaliz must generate these points when evaluating $\sigma$ for the Hilbert basis or Hilbert series. Therefore the determinant sum $\detsum \Sigma=\sum_{\sigma\in\Sigma}\det \sigma$ of $\Sigma$ is a critical complexity parameter. In the following we explain how to optimize it. 

\begin{definition}
Let $G\subset \ZZ^d$ be a finite set. We call the polyhedron $\conv^{\wedge}(G)=\{x\in\RR^d: x=\sum_{g\in G}a_gg,\ a_g\ge 0,\ \sum_{g\in G}a_g\ge 1 \}$ the \emph{upper convex hull} of $G$. The \emph{bottom} $B(G)$ of $G$ is the polyhedral complex of the compact facets of $\conv^{\wedge}(G)$  (or just their union). 
\end{definition}

Let $C$ be the cone generated by $G$. Then $\conv^{\wedge}(G)=\conv(G)+C$, and $B(G)$ is nonempty if and only if $C$ is pointed, or, equivalently, $\conv^{\wedge}(G)$ has a vertex. In this case the bottom is indeed a set of polytopes of dimension $\dim C-1$ since their union is in bijective correspondence with a cross-section of $C$.  Figure \ref{botfig} illustrates the notion of bottom.
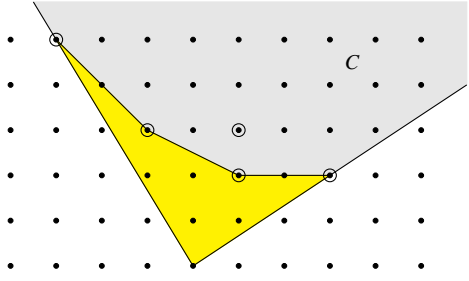
\begin{figure}[hbt]
\begin{center}
	\begin{tikzpicture}[scale=0.6]
	\filldraw[gray!20] (-3.5,5.833) -- (0,0) -- (6,4) -- (6,5.833) -- cycle;
	\filldraw[yellow] (0,0) -- (-3,5) -- (-1,3) -- (1,2) -- (3,2) -- cycle;
	
	\draw (-3,5) -- (-1,3) -- (1,2) -- (3,2);
	
	\draw (-3.5,5.833) -- (0,0) -- (6,4);
	\foreach \x in {-4,...,5}
	\foreach \y in {0,...,5}
	{
		\filldraw[fill=black] (\x,\y)  circle (1.5pt);
	}
	
	\draw (-3,5) circle (4pt) node at (3.5,4.5){\tiny $C$};
	\draw (-1,3) circle (4pt);
	\draw (1,3) circle (4pt);
	\draw (3,2) circle (4pt);
	\draw (1,2) circle (4pt);  
	\end{tikzpicture}
	\end{center}
	\caption{The bottom}\label{botfig}
\end{figure}

As usual, we assume from now on that $C\subset \RR^d$ is pointed and of dimension $d$, and that the monoid $M=C\cap\ZZ^d$ is to be computed.

\begin{definition}
The cones $\RR_+F$ where $F$ runs through the facets in $B(G)$ form the \emph{bottom decomposition} of $C$ with respect to $G$. 

A triangulation $\Sigma$ of $C$ is a \emph{bottom triangulation} with respect to $G$ if every simplicial cone $\sigma\in\Sigma$ is generated by elements of $G\cap F$ where $F$ is a facet of $B(G)$.
\end{definition}

Bottom triangulations are optimal with respect to determinant sum:

\begin{proposition}	Let $\Sigma$ be a bottom triangulation of $C$ with respect to $G$. Then $\detsum(\Sigma)\le\detsum(\Delta)$ for all triangulations $\Delta$ with rays in $G$.
\end{proposition}

\begin{proof}
The union of the basic simplices of $\Sigma$ is the union of the polytopes $\conv(0,F)$ where $F$ runs through the facets of $B(G)$ (see Figure \ref{botfig}). Therefore its determinant sum is the volume of the union $D$ of these polytopes. But  $D$ is contained in the union of the basic simplices of the simplicial cones in $\Delta$, and therefore the volume of $D$ bounds $\detsum \Delta$ from below.	
\end{proof}

Evidently, if the points of $G$ lie in one hyperplane, all triangulations of $C$ with rays through $G$ have the same determinant sum, namely the normalized volume of the polytope $\conv(G,0)$. However, in general the determinant sums can differ widely. Therefore it makes sense to compute a bottom triangulation. First we determine the compact facets of $\conv^{\wedge}(G)$. As usual, let us say that the facet $F$ of the $d$-dimensional polyhedron $Q\subset \RR^d$ is \emph{visible} from $x\in \RR^d$ if $\lambda(x)<0$ for the affine-linear form $\lambda$ defining the hyperplane through $F$ (normed such that $\lambda(y)\ge0$ for $y\in Q$.)

\begin{proposition}\label{botcomp}
Let $F$ be a facet of $\conv^{\wedge}(G)$. Then the following are equivalent:
\begin{enumerate}
\item $F$ belongs to $B(G)$;
\item $F$ is visible from $0$.	
\end{enumerate}
\end{proposition}

\begin{proof}
We choose $\lambda$ as an affine-linear form defining $F$ and a point $x$ of $F$. Let $H$ be the hyperplane spanned by $F$. Suppose first that $\lambda(0)=0$. Then $\lambda$ vanishes on the whole ray from $0$ through $x$, and since this ray belongs to $\conv^{\wedge}(G)$ from $x$ on, it is impossible that $H\cap \conv^{\wedge}(G)$ is compact. The assumption that $\lambda(x)>0$ implies that $\lambda$ has negative values on this ray in points beyond $x$, and this is impossible as well. This proves 1 $\implies$ 2.

Conversely assume that $F$ is visible from $0$, but not compact. Then it is not contained in the compact polytope $P=\conv(G)$. Let $y$ be a point in $F\setminus P$, $y=\sum_{g\in G}a_g g$ with $a=\sum a_g\ge1$, all $a_g\ge 0$. Then $y/a\in P$, and since $\lambda(y)=0$ and $\lambda(y/a)\ge 0$, it follows that $\lambda(0)\ge 0$ since $y/a$ lies between $0$ and $y$. This is a contradiction.
\end{proof}

Normaliz uses lexicographic triangulations (see \cite{BIS}). These are uniquely determined by the order in which the elements are successively added in building the cone. Therefore we can triangulate $\RR_+F$ separately for all bottom facets $F$ using only points in $G\cap F$. These triangulations coincide on the intersections of the cones $\RR_+F$ and can be patched to a triangulation of $\RR_+C$.

Normaliz does not blindly compute triangulations, taking the set $G$ in the order in which it is given. In the presence of a grading it first orders the generating set by increasing degree, and this has already a strong effect on the determinant sum. Nevertheless, bottom decomposition can often improve the situation further.

If the Hilbert basis of $C\cap\ZZ^d$ can be computed quickly by the dual algorithm, one can use it as input for a second run that computes the Hilbert series. (Since version 3.2.0, Normaliz tries to guess whether the primal or the dual algorithm is better for the given input, but the algorithm can also be chosen by the user.) It is clear that bottom decomposition with $G$ being the Hilbert basis, produces the smallest determinant sum of any triangulation of $C$ with rays through integer points. But the Hilbert basis has often many more elements than the set of extreme rays, and this can lead to a triangulation with a much larger number of simplicial cones.  Despite of reducing the determinant sum, it may have a negative effect on computation time. The following example, a Hilbert series computation in social choice theory (input file \texttt{CondEffPlur.in} of the Normaliz distribution; see \cite{BI}, \cite{BIS} or Schürmann \cite{Sch}), demonstrates the effect; see Table \ref{CondEff}.
\begin{table}[hbt]
\setlength{\tabcolsep}{3.2pt}
\renewcommand{\arraystretch}{1.4}
\begin{center}
\begin{tabular}{|r|r|r|r|}\hline
input	& triangulation size& determinant sum& computation time\\ \hline
inequalities & 347,225,775,338 &4,111,428,313,448  & 112:43:17 h\\ \hline
inequalities, \texttt{-b} & 288,509,390,884 & 1,509,605,641,358 & 84:26:19 h\\ \hline
Hilbert basis, \texttt{-b}& 335,331,680,623&1,433,431,230,802& 97:50:05 h \\ \hline
\end{tabular}
\end{center}
\vspace*{2ex}
\caption{Effect of bottom decomposition}\label{CondEff}
\end{table}
With the input ``inequalities'', Normaliz first computes the extreme rays and then applies the primal algorithm to compute the Hilbert series. The option \texttt{-b} forces bottom decomposition. The computation times were taken on a system equipped with $4$ Xeon  E5-2660 at 2.20GHz,  using $30$ parallel threads.

At present Normaliz computes the bottom facets as suggested by Proposition \ref{botcomp}. Since we must homogenize the polyhedron $\conv^{\wedge}(G)$, this amounts to doubling the set $G$ to $G\times\{0\}\cup G\times\{1\}\in \RR^{d+1}$. The advantage of this approach is that one simultaneously computes the facets of $C$ and the bottom facets. Nevertheless, the time spent on this computation can outweigh the saving by a smaller determinant sum. Therefore Normaliz only applies bottom decomposition if asked for by the user or if the bottom is very ``rough''. Roughness is measured by the ratio of the largest degree of a generator and the smallest. At present bottom decomposition is activated if the roughness is $\ge 10$.

We will try to improve the efficiency of bottom decomposition by speeding up its computation. The following proposition suggests a potential approach:

\begin{proposition}\label{bothelp}
With the notation introduced above, let $z\in C$. Then the following are equivalent for a set $F\subset\RR^d$:
\begin{enumerate}
\item $F$ is a facet of $B(G)$;
\item $F$ is a facet of $\conv(G)+\RR_+z$ that is visible from $0$; 
\end{enumerate}
\end{proposition}

The easy proof is left to the reader. If one chooses $z=0$ in Proposition \ref{bothelp}, then one must compute all facets of the polytope $\conv(G)$, not only those in the bottom, but also those in the ``roof'. Choosing $z\ne 0$, for example in the interior of $C$, ``blows the roof off'', and it may be the better choice.

\section{Integral closure as a module}\label{ModGen}

Let $M\subset \ZZ^d$ be a positive affine monoid, $L\supset \gp(M)$ a subgroup of $\ZZ^d$, and $C$ the cone generated by $M$. Then $\overline M_L=C\cap L$ is the integral closure of $M$. It is not only a finitely generated monoid itself, but also a finitely generated $M$-module: there exist $y_1,\dots,y_m\in \overline M_L$ such that $\overline M_L=\bigcup_{i=1}^m y_i+M$. If $M$ (and therefore $\overline M_L$) is positive, then the set $\{y_1,\dots,y_m\}$ is unique once it is chosen minimal. It contains $0$ since $M\subset \overline M_L=C\cap L$.

Geometrically one can interpret the difference $\overline M_L=C\cap L\setminus M$ as the set of ``gaps'' or "holes'' of $M$ in $\overline M_L=C\cap L$, and the nonzero elements of $\{y_1,\dots,y_m\}$ are the ``fundamental holes'' in the terminology of \cite{KLRY}. Since version 3.0.0 Normaliz computes the set $\{y_1,\dots,y_m\}$, and therefore the fundamental holes. 

In the following we assume $L=\ZZ^d$, and set $\widetilde M=\overline M_{\ZZ^d}$. (In \cite{BG} $\overline M$ is reserved for the normalization $\overline M_{\gp(M)}$.)  Evidently the Hilbert basis elements of $\widetilde M$ outside $M$ belong to $\{y_1,\dots,y_m\}$, but in general this set is much larger than the Hilbert basis. Let $M$ be the monoid generated by linearly independent vectors $v_1,\dots,v_d$. Then the lattice points in $\para(v_1,\dots,v_d)$ form a system of module generators of $\widetilde M$, but in general they do not all belong to the Hilbert basis; see Figure \ref{figmod} where $G$ is generated by $(2,1)$ and $(1,3)$. The Hilbert basis elements outside $G$ are only $(1,1) $and $(1,2)$.
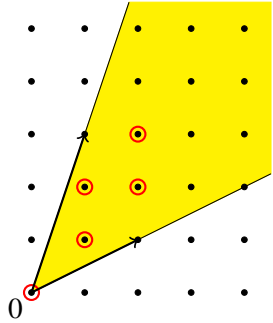
\begin{figure}[hbt]
\begin{center}
	\begin{tikzpicture}[scale=0.7]
	\filldraw[yellow] (0,0) -- (1.833,5.5) -- (4.5,5.5) -- (4.5,2.25) -- cycle;
	\draw (0,0) -- (1.833,5.5);
	\draw (0,0) -- (4.5,2.25) node at (-0.3,-0.3){\small $0$};
	\foreach \x in {0,...,4}
	\foreach \y in {0,...,5}
	{
		\filldraw[fill=black] (\x,\y)  circle (1.5pt);
	}
	\draw[red,thick] (1,1) circle (4pt);
	\draw[red,thick] (2,3) circle (4pt);
	\draw[red,thick] (1,2) circle (4pt);
	\draw[red,thick] (2,2) circle (4pt);
	\draw[red,thick] (0,0) circle (4pt);
	\draw[->,thick] (0,0) -- (1,3);
	\draw[->,thick] (0,0) -- (2,1);
	\end{tikzpicture}
\end{center}
\caption{Module generators of integral closure}\label{figmod}
\end{figure}

Since Normaliz computes the sets $\para(v_1,\dots,v_d)$ for the simplicial cones $\RR_+v_1+\dots+\RR_+v_d$ in a triangulation of $C$ with rays in a given generating set of $M$, it is only a matter of restricting the ``reducers'' in the ``global'' reduction to elements of $G$. 

\begin{proposition}\label{modori}
Let $G\subset \ZZ^d$ generate the positive affine monoid $M\subset\ZZ^d$, and let $\Sigma$ be a triangulation of $C$ with rays in $G$. Then the union $H$ of the sets $\para(\sigma)\cap\ZZ^d$, $\sigma\in\Sigma$ generates the module $\widetilde M$ over $M$.

An element $y\in H$ belongs to the minimal generating set of $\widetilde M$ if and only if $y-x\notin C$, for all $x\in G$, $x\neq 0$.
\end{proposition}

\begin{proof}
Only the second statement may need a justification. We can of course assume that $0\notin G$. Suppose first that $z=y-x\in C$ for some $x\in G$. Then $z\in \widetilde M$ and  $y+ M\subset z+M$ so that $y$ does not belong to the minimal generating set.

Conversely, if $y-x\notin C$ for all $x\in G$, then there is no element $z\in \widetilde M$, $z\neq y$, such that $y\in Z+M$, and so $y$ belongs to the minimal generating set.	
\end{proof}

Normaliz computes minimal sets of module generators not only in the discussed homogeneous case, but also in the inhomogeneous case in which the module is the set of lattice points in a polyhedron $P$ and $G$ generates $\rec(P)$ (since version 3.1.0).

\section{Homogeneous systems of parameters}\label{hsop}

As above, we consider monoids $M=C\cap L$ where $C\subset\RR^d$ is a rational pointed cone and $L\subset \ZZ^d$ is a subgroup. We may right away assume that $d=\dim C$ and $L=\ZZ^d$. Since we want to discuss Hilbert series, we need a grading $\deg:\ZZ^d\to \ZZ$ such that $\deg(x)>0$ for $x\in M$, $x\neq 0$. Additionally we assume that $\deg$ takes the value $1$ on $\gp(M)$, a standardization that Normaliz always performs. The following classical theorem shows that the Hilbert series can be expressed as a rational function.

\begin{theorem}[Ehrhart, Stanley, Hilbert-Serre]\label{Ehr}
\leavevmode
\begin{enumerate}
\item 
The Hilbert series $H_M(t)=\sum_{x\in M} t^{\deg(x)}$ is (the power series expansion of) a rational function that can be written in the form
\begin{equation}
H_M(t)=\frac{Q(t)}{(1-t^\ell)^d}\label{FF}
\end{equation}
where $Q(t)=1+h_1t+\dots+h_st^s$ is a polynomial of degree $s<r\ell$  with nonnegative integer coefficients $h_i$, and $\ell$ is the least common multiple of the degrees of the extreme integral generators of $C$.
\item There exists a (unique) quasipolynomial $q_M(k)$ of degree $r-1$ and period dividing $\ell$ such that $\#\{x\in M: \deg(x)=k\}=q_M(k)$ for all $k>s-r\ell$.
\end{enumerate}
\end{theorem}

It is not difficult to derive the first claim from the existence of a Stanley decomposition so that $H_M(t)$ is a sum of terms given by \eqref{StD}. This explains that all coefficients of the numerator polynomial are nonnegative. There is also an access via commutative algebra which we will explain below.  	

A quasipolynomial of period $\pi>0$ and degree $g$ is a function $q:\ZZ\to\CC$ that can be represented in the form$$
q(k)=q_0^{(k)}+q_1^{(k)}k+\dots+q_g^{(k)}k^g
$$
with $q_i^{(k)}=q_i^{(j)}$ for all $i$ whenever $j\equiv k\pod \pi$; moreover, one has $q_g^{(k)}\neq 0$ for at least one $k$ and $\pi$ is chosen as small as possible. The quasipolynomial in Theorem \ref{Ehr} is called the \emph{Hilbert quasipolynomial} of $M$.

We use the terms ``Hilbert series'' and ``Hilbert quasipolynomial''. One could equally well name these objects after Ehrhart. In fact, the Hilbert series of $M$ is nothing but the Ehrhart series of the polytope that one obtains by intersecting $C$ with the hyperplane of degree $1$ elements in $\RR$.

While Theorem \ref{Ehr} gives a representation of $H_M(t)$ in which all parameters have a natural combinatorial description, it is not completely satisfactory since the denominator often has a very large degree and one can do better. It is our goal to find a representation of $H_M(t)$ as a fraction whose 
\begin{enumerate}
\item denominator is of the form $(1-t^{g_1})\cdots(1-t^{g_d})$ and of small degree  $g_1+\dots+g_d$ and such that
\item the coefficients of the numerator polynomial are nonnegative integers and have a combinatorial interpretation.
\end{enumerate}
We will give an example showing that in general there is no canonical choice of the denominator. Nevertheless it makes sense to search for a good choice. Of course, if all extreme generators have degree $1$, then the denominator of \eqref{FF} is $(1-t)^d$, and there is nothing to discuss.

By default Normaliz proceeds as follows: It reduces the fraction \eqref{FF} to lowest terms and obtains a representation
$$
H_M(t)=\frac{\widetilde Q(t)}{\zeta_{q_1}^{e_1} \cdots \zeta_{q_u}^{e_u}}
$$
with cyclotomic polynomials $\zeta_k$, $1=q_1<q_2<\dots<q_u$. Then it takes $g_d$ as the lcm of all $q_i$, replaces their product by $(1-t^{g_d})$ and proceeds with then remaining cyclotomic factors etc. In this way the $g_k$ express the periods of the coefficients in the Hilbert quasipolynomial: $g_i$ is the lcm of the periods of the coefficients $q_d,\dots,q_{d-i+1}$. We will refer to the denominator of this representation as \emph{standard denominator}. This choice is easy to compute and natural in its way, but not satisfactory if one wants a combinatorial interpretation of the coefficients in the numerator, as the following example shows.

Consider the cone $C=\RR_+(1,2)+\RR_+(2,1)$ with the grading $\deg(x_1,x_2)=x_1+x_2$ (known as the \emph{total grading}). Then Hilbert series with standard denominator is:
$$
H_M(t)=\frac{1-t+t²}{(1-t)(1-t^3)},
$$
with coprime numerator and denominator, and the denominator even has the desired form $(1-t^{g_1})(1-t^{g_2})$. However, the numerator has a negative coefficient.

Commutative algebra suggests us to choose $g_1,\dots,g_d$ as the degrees of the elements in a \emph{homogeneous system of parameters} (hsop for short). Since version 3.1.2 Normaliz can compute such degrees. However, one must use this option with care since it requires the analysis of the face lattice of $C$, an impossible task if $C$ has a large number of facets.

Let $R=\bigoplus_{i=0}^{\infty}R_i$ be a finitely generated $\ZZ$-graded algebra over some infinite field $K=R_0$ of Krull dimension $\dim R=d$. Its graded maximal ideal is given by $\mathfrak{m}=\bigoplus_{i>0}R_i$. In our case, $R$ is the monoid algebra $K[M]$ which is Cohen-Macaulay by a theorem of Hochster's, since $M$ is normal, see \cite[Theorem 6.10]{BG}.

We call homogeneous elements $\theta_1,\ldots,\theta_d\in\mathfrak{m}$ a \emph{homogeneous system of parameters} if $\mathfrak{m}=\mathrm{Rad}(\theta_1,\ldots,\theta_d)$ or, equivalently, $\dim R/\theta=0$, where $\theta=(\theta_1,\ldots,\theta_d)$.

The existence of such a system is guaranteed in the $\ZZ$-graded case by the \emph{prime avoidance lemma}, see \cite[Lemma 6.2]{BG}:
\begin{lemma}
Let $R$ be a $\ZZ$-graded ring and $I\subset R$ an ideal generated in positive degree. Let $\mathfrak{p}_1,\ldots,\mathfrak{p}_r$ be prime ideals such that $I\not\subset\mathfrak{p}_i\,$ for $i=1,\ldots,r$. Then there exists a homogeneous element $x\in I$ with $x\notin\mathfrak{p}_1\cup\dots\cup\mathfrak{p}_r$.
\end{lemma}
 For any ideal $I$ in $R$ generated in positive degree of height $\hht(I)=h$, the lemma provides the existence of elements $\theta_1,\ldots,\theta_h$ such that $\hht(\theta_1,\ldots,\theta_i)=i$ for all $i=1,\ldots,h$.

If $\theta_1,\ldots,\theta_d$ is an hsop for $K[M]$, the Hilbert series can be written in the form
$$
H_M(t)=\frac{h_0+h_1t+\ldots+h_m t^m}{(1-t^{g_1})\cdots(1-t^{g_d})},
$$
where $g_j=\deg\theta_j$. Furthermore $h_i$ counts the number of elements of degree $i$ in a homogeneous basis of $K[M]$ over $K[\theta_1,\ldots,\theta_d]$ and in particular $h_i$ is non-negative (see \cite[Theorem 6.40]{BG}). 

To reach our mentioned goal of finding a nice representation of the Hilbert series, we therefore compute (the degrees of) an hsop for the monoid algebra $K[M]$.

Our main idea for the construction of an hsop is generating elements $\theta_i$ with $\hht(\theta_1,\ldots,\allowbreak \theta_i)=i$ from the extreme integral generators of the cone $C$. We denote them by $x_1,\ldots,x_n\in\ZZ^n$ and note that $\hht(x_1,\ldots,x_n)=d$, where $x_1,\ldots,x_n$ are seen as monomials in $K[M]$. This claim will be justified below.

We successively insert the monomials $x_j$ into a monomial ideal and compute its height. Note that in each step the height of this ideal can only increase by at most one via Krull's principal ideal theorem, see \cite[Theorem A.1]{BH}. If 
$$
\hht(x_1,\ldots,x_j)=i>i-1=\hht(x_1,\ldots,x_{j-1}),
$$
we let 
$$
\theta_i:=\lambda_1x_1^{a_1}+\ldots+\lambda_jx_j^{a_j},
$$
where $\lambda_k\in K$ are generic coefficients and the exponents $a_k$ are chosen in such a way that $\theta_i$ is homogeneous of degree $\lcm(\deg(x_1),\ldots,\deg(x_j))$. We point out that the height does not change if we replace the $x_i$ by powers of them. Furthermore, all current monomials $x_1,\ldots,x_j$ are needed in general to ensure that $\hht(\theta_1,\ldots,\theta_i)=i$.

We are left with the task to compute $\hht(x_1,\ldots,x_j)$. The minimal prime ideals of a monomial ideal $I$ in the monoid algebra $K[M]$ are of the form $\mathfrak{p}_F=K\{M\setminus F\}$, where $F$ runs through all faces of $C$ which are maximal with respect to disjointness to $I$. Furthermore the height of a prime ideal is given by the codimension of its respective face, i.e. $\hht(\mathfrak{p}_F)=d-\dim(F)$ (see for instance \cite[Corollary~4.35 and Proposition~4.36]{BG}). (In particular, the ideal generated by the monomials $x_1,\dots,x_n$ has height $d$: the only face disjoint to them is $\{0\}$.) In conclusion
$$
\hht(x_1,\ldots,x_j)=\min_{F\text{ face}}\left\{\codim(F);F\cap (x_1,\ldots,x_j)=\emptyset\right\}.
$$

These considerations lead to a step-by-step algorithm to compute the \emph{heights vector} $h\in\ZZ_+^n$ with $h_j=\hht(x_1,\ldots,x_j)$, see Algorithm~\ref{alg:heights}.

\begin{algorithm}
	     \caption{Heights}
	     \label{alg:heights}
	    \begin{algorithmic}[1]
		\Let{$h_0$}{$1$}
		\Let{$\mathcal{G}$}{facets of $C$}
		\Let{$m$}{$d$}
		\For{$j=1,\ldots,n$}
		\Let {$\mathcal{G}_1$}{$\{G_k\in\mathcal{G};x_j\notin G_k \}$}
		\Let {$\mathcal{G}_2$}{$\{G_k\in\mathcal{G};x_j\in G_k \}$}
		\If{$\mathcal{G}_1\neq\emptyset$}{\If{$\max_{G_k\in\mathcal{G}_1}\{\dim(G_k)\}<m$} $m\gets m-1;\; h_j=h_{j-1}+1$ \Else { $h_j=h_{j-1}$} \EndIf} \Else{ $h_j=h_{j-1}+1$} \EndIf
		\ForAll{facets $F_{\ell}$ with $x_j\notin F_{\ell}$}
			\ForAll{$G_k\in\mathcal{G}_2$}  \Let{$G_{k,\ell}$}{$G_k\cap F_{\ell}$} \EndFor
		\EndFor
		\Let{$\mathcal{G}$}{$\mathcal{G}_1\cup \{\text{maximal faces from } G_{k,\ell}\}$} 
		\EndFor
	          
	    \end{algorithmic}
\end{algorithm}

Some of the facets can be neglected in the process of taking intersections with the faces in step $j$ due to the following criteria:
\begin{enumerate}
\item The facet contains the current generator $x_j$;
\item The facet only involves generators appearing in faces in $\mathcal{G}_1$ or $x_1,\ldots,x_{j-1}$;
\item Facets only involving the generators $x_1,\ldots,x_j$ can be ignored for all following iterations.
\end{enumerate}

Once the heights vector $h$ is computed, the degrees of the corresponding hsop can be determined as mentioned before, although not all initial generators need to appear in the $\lcm$ to compute the homogeneous degree. More precisely, let $\ell$ denote the smallest index such that $h_\ell=h_{\ell+1}$. Since $\hht(x_1,\ldots,x_j,x_{j+1})=h_{j+1}=h_{j}+1=\hht(x_1,\ldots,x_j)+1$ for $j=1,\ldots,\ell-1$ we have
$$
\deg(\theta_i)=\begin{cases}  \deg(x_i), & \text{if }i\leq \ell   \\ \mathrm{lcm}(\deg(x_{\ell+1}),\ldots,\deg(x_i)), & \text{if }i>\ell.   \end{cases}
$$

We finally calculate the numerator of the new representation of the Hilbert series, by multiplying the form with cyclotomic polynomials in the denominator with the product $(1-t^{g_1})\dots(1-t^{g_d}),$ where $g_j=\deg(\theta_j)$.

We note that for the simplicial case the extreme integral generators $x_1,\ldots,x_d$ already form an hsop. Therefore the choice of their degrees in the denominator of the Hilbert series can be considered a canonical. In the above simplicial example $C=\RR_+x_1+\RR_+x_2$ with $x_1=(1,2)$ and $x_2=(2,1)$ the series can be expressed as:
$$
H_M(t)=\frac{1+t^2+t^4}{(1-t^3)^2},
$$
where the degrees appearing in the denominator come from the extreme integral generators of $C$. The numerator has non-negative coefficients and counts the number of homogeneous basis elements of $K[M]$ as a $K[x_1,x_2]$-module per degree, in this case $(0,0), (1,1)$ and $(2,2)$ of degree $0, 2$ and $4$ respectively. This example also shows that using the Hilbert basis instead of the extreme integral generators as a generating system for $M$ sometimes yield smaller exponents in the denominator, namely $(1-t^2)(1-t^3)$. However, using the Hilbert basis for the algorithm increases the complexity of taking intersections remarkably, which is the most expensive step. 

As an example, let $C=Q\times\{1\}$ be the cone over a square $Q$, see Figure~\ref{fig:square}. The degree is given by $\deg(x_i)=i$ for $i=1,\ldots,4$. (This choice is eligible since the only condition for this configuration is that the two sums of the degrees of antipodal points agree.) We get the following sequence of heights, which is also illustrated in Figure~\ref{fig:square} where dotted lines indicate the maximal disjoint faces:
\begin{multline*}
h_1=\hht(x_1)=1,\; h_2=\hht(x_1,x_2)=1,\; h_3=\hht(x_1,x_2,x_3)=2,\;\\ h_4=\hht(x_1,x_2,x_3,x_4)=3.
\end{multline*}

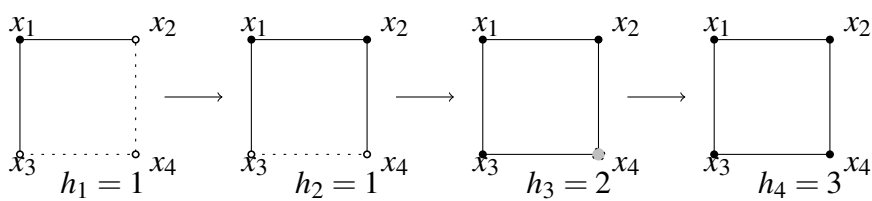
\begin{figure}[h]
\tikzstyle{ipe stylesheet} = [
	x=0.8bp, y=0.8bp,
	ipe fdisk/.pic={
	    \filldraw[line width=0.2*3]
	      (0,0) circle[radius=0.5*3];
	    \coordinate () at (0,0);
	  },
	ipe dash dotted/.style={dash pattern=on 1bp off 3bp}
]
\begin{tikzpicture}[scale=0.95,ipe stylesheet]
  \draw[shift={(257.73, 760.299)}, scale=0.8869]
    (0, 0)
     -- (64, 0);
  \node[anchor=base west]
     at (45.089, 741.204) {$h_1=1$};
  \node[anchor=base west]
     at (160.231, 741.204) {$h_2=1$};
  \draw[shift={(101.643, 788.678)}, scale=0.8869, ->]
    (0, 0)
     -- (32, 0);
  \node[anchor=base west]
     at (273.749, 741.204) {$h_3=2$};
  \draw[shift={(215.161, 788.678)}, scale=0.8869, ->]
    (0, 0)
     -- (32, 0);
  \draw[shift={(314.489, 817.058)}, scale=0.8869]
    (0, 0)
     -- (0, -64);
  \draw[shift={(257.73, 760.299)}, scale=0.8869]
    (0, 0)
     -- (0, 64)
     -- (64, 64);
  \pic[fill=black]
     at (257.7305, 817.0576) {ipe fdisk};
  \node[anchor=base west]
     at (247.499, 821.387) {$x_1$};
  \node[anchor=base west]
     at (316.238, 821.387) {$x_2$};
  \node[anchor=base west]
     at (247.498, 752.302) {$x_3$};
  \node[anchor=base west]
     at (316.238, 752.302) {$x_4$};
  \pic[fill=black]
     at (314.4894, 817.0576) {ipe fdisk};
  \pic[fill=black]
     at (257.7305, 760.2987) {ipe fdisk};
  \filldraw[ipe dash dotted, fill=lightgray]
    (314.4894, 760.2987) circle[radius=3.0402];
  \draw[shift={(144.213, 760.299)}, scale=0.8869, ipe dash dotted]
    (0, 0)
     -- (64, 0);
  \draw[shift={(200.972, 817.058)}, scale=0.8869]
    (0, 0)
     -- (0, -64);
  \draw[shift={(144.213, 760.299)}, scale=0.8869]
    (0, 0)
     -- (0, 64)
     -- (64, 64);
  \pic[fill=black]
     at (144.2126, 817.0576) {ipe fdisk};
  \pic[fill=white]
     at (144.2126, 760.2987) {ipe fdisk};
  \pic[fill=white]
     at (200.9715, 760.2987) {ipe fdisk};
  \node[anchor=base west]
     at (133.981, 821.387) {$x_1$};
  \node[anchor=base west]
     at (202.72, 821.387) {$x_2$};
  \node[anchor=base west]
     at (133.981, 752.302) {$x_3$};
  \node[anchor=base west]
     at (202.72, 752.302) {$x_4$};
  \pic[fill=black]
     at (200.9715, 817.0576) {ipe fdisk};
  \draw[shift={(30.695, 760.299)}, scale=0.8869]
    (0, 0)
     -- (0, 64)
     -- (64, 64);
  \draw[shift={(30.695, 760.299)}, scale=0.8869, ipe dash dotted]
    (0, 0)
     -- (64, 0)
     -- (64, 64);
  \pic[fill=black]
     at (30.6947, 817.0576) {ipe fdisk};
  \pic[fill=white]
     at (87.4536, 817.0576) {ipe fdisk};
  \pic[fill=white]
     at (30.6947, 760.2987) {ipe fdisk};
  \pic[fill=white]
     at (87.4536, 760.2987) {ipe fdisk};
  \node[anchor=base west]
     at (20.463, 821.387) {$x_1$};
  \node[anchor=base west]
     at (89.202, 821.387) {$x_2$};
  \node[anchor=base west]
     at (20.463, 752.302) {$x_3$};
  \node[anchor=base west]
     at (89.203, 752.302) {$x_4$};
  \draw[shift={(371.248, 760.299)}, scale=0.8869]
    (0, 0)
     -- (64, 0);
  \node[anchor=base west]
     at (387.267, 741.204) {$h_4=3$};
  \draw[shift={(328.679, 788.678)}, scale=0.8869, ->]
    (0, 0)
     -- (32, 0);
  \draw[shift={(428.007, 817.058)}, scale=0.8869]
    (0, 0)
     -- (0, -64);
  \draw[shift={(371.248, 760.299)}, scale=0.8869]
    (0, 0)
     -- (0, 64)
     -- (64, 64);
  \pic[fill=black]
     at (371.2484, 817.0576) {ipe fdisk};
  \node[anchor=base west]
     at (361.017, 821.387) {$x_1$};
  \node[anchor=base west]
     at (429.756, 821.387) {$x_2$};
  \node[anchor=base west]
     at (361.016, 752.302) {$x_3$};
  \node[anchor=base west]
     at (429.756, 752.302) {$x_4$};
  \pic[fill=black]
     at (428.0074, 817.0576) {ipe fdisk};
  \pic[fill=black]
     at (371.2484, 760.2987) {ipe fdisk};
  \pic[fill=black]
     at (428.0074, 760.2987) {ipe fdisk};
\end{tikzpicture}
\caption{Sequence of heights for a cone over a square}
\label{fig:square}
\end{figure}

The degrees for the corresponding hsop are given by $\deg(\theta_1)=1, \deg(\theta_2)=6$ and $\deg(\theta_3)=12$ and the Hilbert series has the form
$$
H_M(t)=\frac{1+t^2+t^3+2t^4+2t^6+t^7+2t^8+2t^{10}+t^{11}+t^{12}+t^{14}}{(1-t)(1-t^6)(1-t^{12})}.
$$

The heights vector and the degrees of the corresponding hsop can also be seen on the terminal if Normaliz is run with the verbosity option:
\begin{verbatim}
Heights vector: 1 1 2 3

Degrees of HSOP: 1 6 12 
\end{verbatim}

The Hilbert series with standard denominator for this cone is
$$
H_M(t)=\frac{1+t^3+t^4-t^5+t^6+t^7+t^{10}}{(1-t)(1-t^2)(1-t^{12})},
$$
which again has a negative coefficients in the numerator.

If the order of the generators would be $x_2,x_3,x_1,x_4$ the degrees and hence the exponents in the denominator of the Hilbert series are smaller, namely $\deg(\theta_1)=2, \deg(\theta_2)=3, \deg(\theta_3)=4$ and
$$
H_M(t)=\frac{1+t+t^2+t^3+t^4}{(1-t^2)(1-t^3)(1-t^4)}.
$$

However, considerations about the best possible order of generators would involve knowledge about the algebraic structure and defining equations (in this case $x_1x_4=x_2x_3$) of the input data, which are not accessible in Normaliz. Moreover, there is no clear answer to the question what an optimal choice for the exponents in the denominator should look like. Nevertheless, a possibility to improve the current representation would be a dynamic choice of the generators, where the next generator is chosen to lie in as many faces as possible, e.g. $x_1,x_4,x_2,x_3$ in the above example. Future versions of Normaliz may contain this choice.
%
%
%
%
%
%
%

\section{Class group}\label{class}

The monoids $M=C\cap L$ where $C\subset \RR^d$ is a rational cone and $L$ a subgroup of $\ZZ^d$ are exactly the normal affine monoids. For such a monoid $M$ and a field $K$ the monoid algebra $K[M]$ is a normal Noetherian domain, which has a divisor class group $\Cl(K[M])$, the group of isomorphism classes of divisorial ideals. It is not hard to prove that every isomorphism class is represented by a monomial divisorial ideal, and if one analyzes which monomial ideals are divisorial and when two such ideals are isomorphic modules, then one obtains \emph{Chouinard's theorem}, see \cite[Corollary 4.56]{BG}: 
\begin{theorem}
Let $\sigma:\gp(M)\to \ZZ^s$ be the standard map. Then the divisor class group $\Cl(K[M])$ (identical to the divisor class group $\Cl(M)$ of $M$) is given by $\ZZ^s/\sigma(\gp(M))$.
\end{theorem}

If $\dim C=d$ and $L=\ZZ^d$, one has $\gp(M)=\ZZ^d$. Therefore $\Cl(M)=\ZZ^s/\sigma(\ZZ^d)$. Since $\sigma$ is known, the computation of the divisor class group is a cheap by-product. Let $A$ be the matrix whose columns are the support forms with coordinates in the dual basis to the unit vectors in $\ZZ^d$. Then the rows generate $\sigma(\ZZ^d)\subset \ZZ^s$, and it is only a matter of computing the Smith normal form of $A$. It immediately yields a decomposition $\Cl(M)=\ZZ^r\oplus (\ZZ/c_1\ZZ)^{e_1}\oplus\dots\oplus (\ZZ/c_u\ZZ)^{e_u}$ such that $c_1\mid \cdots \mid c_u$.

\end{document}